\documentclass[10pt,reqno]{amsart}

\usepackage{amssymb}
\usepackage{amsmath}
\usepackage{amsfonts}
\usepackage{amscd}
\usepackage{mathrsfs}
\usepackage{color}

\newtheorem{theorem}{Theorem}

\newtheorem{example}[theorem]{Example}
\newtheorem{prop}[theorem]{Proposition}

\newcommand{\C}{\mathbf{C}}
\newcommand{\n}[1]{\left\vert#1\right\vert}

\begin{document}
\title[]{On the Equation $f^n(z)+g^n(z)=e^{\alpha z+\beta}$}

\author[]{Qi Han and Feng L\"{u}}

\address{Department of Mathematics, Texas A\&M University at San Antonio
\vskip 2pt San Antonio, Texas 78224, USA \hspace{21.6mm}{\sf Email: qhan@tamusa.edu}
\vskip 6pt College of Science, China University of Petroleum at Qingdao
\vskip 2pt Qingdao, Shandong 266580, P.R. China \hspace{11mm}{\sf lvfeng18@gmail.com}}



\begin{abstract}
We describe meromorphic solutions to the equations $f^n(z)+\left(f'\right)^n(z)=e^{\alpha z+\beta}$ and $f^n(z)+f^n(z+c)=e^{\alpha z+\beta}$ ($c\neq0$) over the complex plane $\C$ for integers $n\geq1$.
\end{abstract}

\maketitle





This paper is devoted to the description of meromorphic solutions to the following functional equation
\begin{equation}\label{Eq1}
f^n(z)+g^n(z)=e^{\alpha z+\beta},
\end{equation}
where $g(z)=f'(z)$ or $g(z)=f(z+c)$ for $\alpha,\beta,c(\neq0)\in\C$, when $n\geq1$.

In particular, when $\alpha=\beta=0$, then \eqref{Eq1} is reduced to the following well-known Fermat-type functional equation, initialed by Gross \cite{Gr1,Gr2,Gr3} and Baker \cite{Ba},
\begin{equation}\label{Eq1*}
f^n(z)+g^n(z)=1.
\end{equation}
Below, we summarize all possible solutions to \eqref{Eq1*} (see theorem 2.3 in Han \cite{Ha}).

\begin{prop}\label{Pr1}
For nonconstant meromorphic solutions $f$ and $g$ to the functional equation \eqref{Eq1*}, one has that
{\bf(A)} when $n=2$, the only solutions are $f=\frac{2\omega}{1+\omega^2}$ and $g=\frac{1-\omega^2}{1+\omega^2}$ for a nonconstant meromorphic function $\omega$;
{\bf(B)} when $n=3$, the only solutions are $f=\frac{1}{2\mathfrak{p}(h)}\left(1+\frac{\sqrt{3}}{3}\mathfrak{p}'(h)\right)$ and $g=\frac{\eta}{2\mathfrak{p}(h)}\left(1-\frac{\sqrt{3}}{3}\mathfrak{p}'(h)\right)$ for a nonconstant entire function $h$ and a cube root $\eta$ of unity, where $\mathfrak{p}$ denotes the Weierstrass $\mathfrak{p}$-function;
{\bf(C)} when $n\geq4$, there are no such solutions.
\end{prop}

Via $\omega=\tan\left(\frac{h}{2}\right)$, we have $f=\frac{2\omega}{1+\omega^2}=\sin(h)$ and $g=\frac{1-\omega^2}{1+\omega^2}=\cos(h)$ in case {\bf(A)} for $n=2$ are the only entire solutions to the functional equation \eqref{Eq1*} for an entire function $h$.
Moreover, $\mathfrak{p}(z)$, the Weierstrass elliptic $\mathfrak{p}$-function with periods $\omega_1$ and $\omega_2$, is defined to be
\begin{equation}
\mathfrak{p}(z;\omega_1,\omega_2):=\frac{1}{z^2}+\sum_{\mu,\nu\in\mathbf{Z};\hspace{0.2mm}\mu^2+\nu^2\neq0}
\left\{\frac{1}{\left(z+\mu\omega_1+\nu\omega_2\right)^2}-\frac{1}{\left(\mu\omega_1+\nu\omega_2\right)^2}\right\},\nonumber
\end{equation}
which is an even function and satisfies, after appropriately choosing $\omega_1$ and $\omega_2$,
\begin{equation}\label{Eq2}
(\mathfrak{p}')^2=4\mathfrak{p}^3-1.
\end{equation}

For meromorphic solutions of partial differential equations similar to \eqref{Eq1}, we refer the reader to Li \cite{Li1,Li2}, Chang and Li \cite{CL}, Han \cite{Ha}, and the references therein.

Below, we assume the familiarity with the basics of Nevanlinna theory \cite{Ne} of meromorphic functions in $\C$ such as the first and second main theorems, and the standard notations such as the characteristic function $T(r,f)$, the proximity function $m(r,f)$, and the counting functions $N(r,f)$ (counting multiplicity) and $\bar{N}(r,f)$ (ignoring multiplicity).
$S(r,f)$ denotes a quantity satisfying $S(r,f)=o\left(T(r,f)\right)$ as $r\to\infty$, except possibly on a set of finite logarithmic measure which is not necessarily the same at each occurrence.

First, consider meromorphic solutions to $f^n+\left(f'\right)^n=\gamma^n$ when $n\geq4$ and $\gamma\not\equiv0$.
According to the conclusions of proposition \ref{Pr1}, both $\frac{f}{\gamma}$ and $\frac{f'}{\gamma}$ are constant.
Assume $f=c_1\gamma$ and $f'=c_2\gamma$ to see $c_1\gamma'=c_2\gamma$ with $c^n_1+c^n_2=1$.
If $c_1=0$, $f\equiv0$ and thus $\gamma\equiv0$.
So, $c_1\neq0$, and if $c_2=0$, $f$ is a constant and so is $\gamma$.
When $c_1c_2\neq0$, then $\gamma$ cannot have zeros and poles, and one sees $\gamma^n(z)=e^{\alpha z+\beta}$ with $\alpha=n\frac{c_2}{c_1}$.
This is another reason why we focus on $e^{\alpha z+\beta}$.

Next, for $f^3+\left(f'\right)^3=e^{\alpha z+\beta}$, $f$ must be entire and thus both $\frac{f}{\gamma}$ and $\frac{f'}{\gamma}$ are constant, so that the same conclusion holds as above.
Now, for $f^2+\left(f'\right)^2=e^{\alpha z+\beta}$, $f$ must again be entire and $f(z)=e^{\frac{\alpha z+\beta}{2}}\sin(h(z))$ and $f'(z)=e^{\frac{\alpha z+\beta}{2}}\cos(h(z))$ by proposition \ref{Pr1}, so that $\frac{\alpha}{2}\tan(h)\equiv1-h'$.
As $T(r,h')=O\left(T(r,h)\right)+S(r,h)$ and $\lim\limits_{r\to\infty}\frac{T(r,\tan(h))}{T(r,h)}=+\infty$ (see Clunie \cite[theorem 2 (i)]{Cl} that extended P\'{o}lya \cite{Po}), we see that either $\alpha=0$ and $h'=1$, or $h$ is a constant.

Summarizing the preceding discussions leads to the following result.

\begin{theorem}\label{Th2}
Solutions $f$ to the following differential equation
\begin{equation}\label{Eq3}
f^n(z)+\left(f'\right)^n(z)=e^{\alpha z+\beta}
\end{equation}
must be entire and are such that
{\bf(A)} when $n=1$, the general solutions are $f(z)=\frac{e^{\alpha z+\beta}}{\alpha+1}+ae^{-z}$ for $\alpha\neq-1$ and $f(z)=ze^{-z+\beta}+ae^{-z}$;
{\bf(B)} when $n=2$, either $\alpha=0$ and the general solutions are $f(z)=e^{\frac{\beta}{2}}\sin(z+b)$, or $f(z)=de^{\frac{\alpha z+\beta}{2}}$;
{\bf(C)} finally when $n\geq3$, then the general solutions are $f(z)=de^{\frac{\alpha z+\beta}{n}}$.
Here, $\alpha,\beta,a,b,d\in\C$ with $d^n\left(1+\left(\frac{\alpha}{n}\right)^n\right)=1$ for $n\geq1$.
\end{theorem}

Note when $n\geq2$, equation \eqref{Eq3} may have no solution for $\alpha=ne^{\frac{\left(2k+1\right)\pi i}{n}}$, $k=0,1,\ldots,n-1$.
In addition, we refer the reader to Li \cite{Li3} for some related interesting results.

Now, consider meromorphic solutions $f(z)$ to the following difference equation, with $c\neq0$,
\begin{equation}\label{Eq4}
f^n(z)+f^n(z+c)=e^{\alpha z+\beta}.
\end{equation}
When $n\geq1$, take $f(z)=c_1e^{\frac{\alpha z+\beta}{n}}$ and $f(z+c)=c_2e^{\frac{\alpha z+\beta}{n}}$ to see $c_1e^{\frac{\alpha c}{n}}=c_2$ with $c^n_1+c^n_2=1$, inspired by case {\bf(C)} of proposition \ref{Pr1}.
Note that $c_1c_2\neq0$ and $f(z+c)=e^{\frac{\alpha c}{n}}f(z)$.
As a result, all those \textsl{trivial} solutions are $f(z)=de^{\frac{\alpha z+\beta}{n}}$ with $d^n(1+e^{\alpha c})=1$ for $n\geq1$.

Next, we discuss the existence of nontrivial solutions to \eqref{Eq4} when $n=3${\footnote{Please be reminded that a similar yet simpler approach has been applied for the discussion of meromorphic solutions to $f^3(z)+f^3(z+c)=1$ in Han and L\"{u} \cite{HL}.}}.

\begin{theorem}\label{Th3}
There is no solution of finite order to the following difference equation
\begin{equation}\label{Eq4*}
f^3(z)+f^3(z+c)=e^{\alpha z+\beta}.
\end{equation}
Here, the \textsl{order} of $f$ is defined to be $\rho(f):=\limsup\limits_{r\to+\infty}\frac{\log T(r,f)}{\log r}$.
\end{theorem}

\begin{proof}
Via proposition \ref{Pr1}, one has
\begin{equation}\label{Eq5}
f(z)=\frac{1}{2}\frac{\left\{1+\frac{\sqrt{3}}{3}\mathfrak{p}'(h(z))\right\}}{\mathfrak{p}(h(z))}e^{\frac{\alpha z+\beta}{3}}\hspace{2mm}\mathrm{and}\hspace{2mm}
f(z+c)=\frac{\eta}{2}\frac{\left\{1-\frac{\sqrt{3}}{3}\mathfrak{p}'(h(z))\right\}}{\mathfrak{p}(h(z))}e^{\frac{\alpha z+\beta}{3}}.
\end{equation}
Thus, a routine computation leads to
\begin{equation}\label{Eq6}
\frac{\eta\left\{1-\frac{\sqrt{3}}{3}\mathfrak{p}'(h(z))\right\}}{\mathfrak{p}(h(z))}
=\frac{\left\{1+\frac{\sqrt{3}}{3}\mathfrak{p}'(h(z+c))\right\}}{\mathfrak{p}(h(z+c))}e^{\frac{\alpha c}{3}}.
\end{equation}

Assume $\rho(f)<\infty$.
Then, from \eqref{Eq2} and the first equality in \eqref{Eq5}, one has
\begin{equation}\label{Eq7}
\frac{3f^2(z)\mathfrak{p}^2(h(z))}{e^{\frac{2}{3}\left(\alpha z+\beta\right)}}-\frac{3f(z)\mathfrak{p}(h(z))}{e^{\frac{1}{3}\left(\alpha z+\beta\right)}}+1=\mathfrak{p}^3(h(z)).
\end{equation}

Notice the estimate (2.7) of Bank and Langley \cite{BL} says that
\begin{equation}\label{Eq8}
T(r,\mathfrak{p})=\frac{\pi}{A}\,r^2\left(1+o\left(1\right)\right)\hspace{2mm}\mathrm{and}\hspace{2mm}\rho(\mathfrak{p})=2.
\end{equation}
Here, $A$ is the area of the parallelogram $\mathfrak{S}$ with vertices $0,\omega_1,\omega_2,\omega_1+\omega_2$.

Recall $T(r,e^{\alpha z})=\frac{\n{\alpha}}{\pi}\,r\left(1+o\left(1\right)\right)$.
We combine \eqref{Eq7} and \eqref{Eq8} to observe
\begin{equation}\label{Eq9}
T(r,\mathfrak{p}(h))\leq2T(r,f)+\frac{2}{3}T(r,e^{\alpha z})+O\left(1\right),
\end{equation}
and hence $\rho(\mathfrak{p}(h))<\infty$ as well.
By corollary 1.2 of Edrei and Fuchs \cite{EF} (see also theorem 1 of Bergweiler \cite{Be} for a different and elegant proof), $h$ must be a polynomial.

A side note here is $T(r,\mathfrak{p}(h))=O\left(r^{2l}\right)$ for some positive integer $l\geq1$.

Notice when $\mathfrak{p}(z_0)=0$, then $(\mathfrak{p}')^2(z_0)=-1$ by \eqref{Eq2}.
Now, write all the zeros of $\mathfrak{p}$ by $\left\{z_j\right\}_{j=1}^\infty$ that satisfy $\n{z_j}\to\infty$ as $j\to\infty$, and assume that $h(a_{j,k})=z_j$ for $k=1,2,\ldots,\deg(h)$.
Then, we have $(\mathfrak{p}')^2(h(a_{j,k}))=(\mathfrak{p}')^2(z_j)=-1$.

Suppose there is a subsequence of $\left\{a_{j,k}\right\}_{j=1}^{\infty}$ with respect to $j$ such that $\mathfrak{p}(h(a_{j,k}+c))=0$.
Denote this subsequence still by $\left\{a_{j,k}\right\}_{j=1}^{\infty}$ and without loss of generality fix the index $k$ below.
So, $(\mathfrak{p}')^2(h(a_{j,k}+c))=-1$.
Differentiate \eqref{Eq6} and use substitution to derive
\begin{equation}
\begin{split}
&\,\eta\left\{1-\frac{\sqrt{3}}{3}\mathfrak{p}'(h(a_{j,k}))\right\}\mathfrak{p}'(h(a_{j,k}+c))h'(a_{j,k}+c)\\
=&\left\{1+\frac{\sqrt{3}}{3}\mathfrak{p}'(h(a_{j,k}+c))\right\}\mathfrak{p}'(h(a_{j,k}))h'(a_{j,k})\,e^{\frac{\alpha c}{3}},\nonumber
\end{split}
\end{equation}
from which we observe one and only one of the following situations appears
\begin{equation}
\left\{\begin{array}{ll}
\eta\left\{1-i\frac{\sqrt{3}}{3}\right\}h'(a_{j,k}+c)=\left\{1+i\frac{\sqrt{3}}{3}\right\}h'(a_{j,k})\,e^{\frac{\alpha c}{3}},\\ \\
\eta\,h'(a_{j,k}+c)=-\,h'(a_{j,k})\,e^{\frac{\alpha c}{3}},\\ \\
\eta\left\{1+i\frac{\sqrt{3}}{3}\right\}h'(a_{j,k}+c)=\left\{1-i\frac{\sqrt{3}}{3}\right\}h'(a_{j,k})\,e^{\frac{\alpha c}{3}}.
\end{array}\right.\nonumber
\end{equation}

As $h(z)$ and $h(z+c)$ are polynomials of the same leading coefficient, and there are infinitely many $a_{j,k}$'s with $\n{a_{j,k}}\to\infty$ when $j\to\infty$, we would have to conclude
\begin{equation}\label{Eq10}
\left\{\begin{array}{ll}
\eta\left\{1-i\frac{\sqrt{3}}{3}\right\}h'(z+c)=\left\{1+i\frac{\sqrt{3}}{3}\right\}h'(z)\,e^{\frac{\alpha c}{3}},\\ \\
\eta\,h'(z+c)=-\,h'(z)\,e^{\frac{\alpha c}{3}},\\ \\
\eta\left\{1+i\frac{\sqrt{3}}{3}\right\}h'(z+c)=\left\{1-i\frac{\sqrt{3}}{3}\right\}h'(z)\,e^{\frac{\alpha c}{3}}.
\end{array}\right.
\end{equation}
This is possible only if $\alpha,c$ satisfy $e^{\frac{\alpha c}{3}}=-1,\frac{1}{2}\pm i\frac{\sqrt{3}}{2}$ because $\eta=1,-\frac{1}{2}\pm i\frac{\sqrt{3}}{2}$.

When this is true, one has uniformly by \eqref{Eq10} that $h(z)=az+b$ for $ac\neq0$.
As $\mathfrak{p}(z)$ has two distinct zeros in $\mathfrak{S}$ and thus in each associated lattice, we observe that all the zeros $\left\{z_j\right\}_{j=1}^\infty$ of $\mathfrak{p}(z)$ are transferred to each other through (an integral multiple of) $ac$.
We may for simplicity consider two cases where either $ac=\omega_1,\omega_2,\omega_1+\omega_2$, or $ac\neq\omega_1,\omega_2,\omega_1+\omega_2$ and $ac\in\mathfrak{S}$.
The former cannot occur in view of \eqref{Eq6} using the periodicity of $\mathfrak{p}$ and $\mathfrak{p}'$, while the latter neither - noting $\mathfrak{p}(z)$ has a unique double pole in each lattice, we substitute $z_{\infty}=-\frac{b}{a}$ into \eqref{Eq6} to deduce a contradiction
\begin{equation}
\infty=\frac{\eta\left\{1-\frac{\sqrt{3}}{3}\mathfrak{p}'(0)\right\}}{\mathfrak{p}(0)}
=\frac{\left\{1+\frac{\sqrt{3}}{3}\mathfrak{p}'(ac)\right\}}{\mathfrak{p}(ac)}e^{\frac{\alpha c}{3}}<\infty.\nonumber
\end{equation}

Thus, $\mathfrak{p}(h(a_{j,k}+c))=0$ may only occur for finitely many $a_{j,k}$'s.
Without loss of generality, assume $\mathfrak{p}(h(a_{j,k}+c))\neq0$ for each $k=1,2,\ldots,\deg(h)$ and all $j>J$, with $J$ being a sufficiently large positive integer.
Since $\mathfrak{p}(h(a_{j,k}))=0$ and $(\mathfrak{p}')^2(h(a_{j,k}))=-1$, one has $\mathfrak{p}(h(a_{j,k}+c))=\infty$ when $j>J$ by \eqref{Eq6} again.
As a consequence, noticing $O\left(\log r\right)=S(r,\mathfrak{p}(h))$, we have
\begin{equation}\label{Eq11}
\begin{split}
&N\left(r,\frac{1}{\mathfrak{p}(h(z))}\right)\leq\bar{N}\left(r,\frac{1}{\mathfrak{p}(h(z))}\right)+2N\left(r,\frac{1}{h'(z)}\right)\\
\leq\,\,&\bar{N}(r,\mathfrak{p}(h(z+c)))+2T(r,h')+O\left(\log r\right)\leq\bar{N}(r,\mathfrak{p}(h(z+c)))+S(r,\mathfrak{p}(h)).
\end{split}
\end{equation}

Recall the first equality in \eqref{Eq5} and estimate \eqref{Eq8}.
One has
\begin{equation}\label{Eq12}
T(r,f)\leq T(r,\mathfrak{p}(h))+T(r,\mathfrak{p}'(h))+\frac{1}{3}T(r,e^{\alpha z})+O\left(1\right)\leq O\left(T(r,\mathfrak{p}(h))\right),
\end{equation}
so that $\rho(f)=\rho(\mathfrak{p}(h))$ and $S(r,f)=S(r,\mathfrak{p}(h))$ from \eqref{Eq9} and the side note after it.

Thus, $T(r,e^{\alpha z})=S(r,f)$.
As all the zeros of $f-e^{\frac{\alpha z+\beta}{3}}$, $f-\eta e^{\frac{\alpha z+\beta}{3}}$ and $f-\eta^2e^{\frac{\alpha z+\beta}{3}}$ ($\eta\neq1$) are of multiplicities at least $3$ from \eqref{Eq4*}, Yamanoi's second main theorem \cite{Ya} yields
\begin{equation}
\begin{aligned}
&2T(r,f)\leq\sum_{m=1}^3\bar{N}\left(r,\frac{1}{f-\eta^me^{\frac{\alpha z+\beta}{3}}}\right)+\bar{N}(r,f)+S(r,f)\\
\leq\,\,&\frac{1}{3}\sum_{m=1}^3N\left(r,\frac{1}{f-\eta^me^{\frac{\alpha z+\beta}{3}}}\right)+N(r,f)+S(r,f)\leq2T(r,f)+S(r,\mathfrak{p}(h)).\nonumber
\end{aligned}
\end{equation}
Therefore, one derives $T(r,f)=N(r,f)+S(r,\mathfrak{p}(h))$ so that $m(r,f)=S(r,\mathfrak{p}(h))$.

Finally, applying the lemma of logarithmic derivative, we have
\begin{equation}\label{Eq13}
m\left(r,\frac{1}{\mathfrak{p}(h)}\right)
\leq m(r,f)+m\left(r,\frac{\mathfrak{p}'(h)h'}{\mathfrak{p}(h)}\right)+S(r,\mathfrak{p}(h))=S(r,\mathfrak{p}(h))
\end{equation}
again through the first equality in \eqref{Eq5}.
Combining \eqref{Eq11} and \eqref{Eq13} leads to
\begin{equation}\label{Eq14}
\begin{aligned}
&T(r,\mathfrak{p}(h))+O\left(1\right)=T\left(r,\frac{1}{\mathfrak{p}(h)}\right)
=m\left(r,\frac{1}{\mathfrak{p}(h(z))}\right)+N\left(r,\frac{1}{\mathfrak{p}(h(z))}\right)\\
\leq\,\,&\bar{N}(r,\mathfrak{p}(h(z+c)))+S(r,\mathfrak{p}(h))\leq\frac{1}{2}N(r,\mathfrak{p}(h(z+c)))+S(r,\mathfrak{p}(h))\\
\leq\,\,&\frac{1}{2}T(r,\mathfrak{p}(h(z+c)))+S(r,\mathfrak{p}(h))\leq\frac{1}{2}T(r,\mathfrak{p}(h))+S(r,\mathfrak{p}(h)),
\end{aligned}
\end{equation}
where theorem 2.1 of Chiang and Feng \cite{CF} was applied.
This is a contradiction.
\end{proof}

\begin{example}\label{Ex4}
Assume $f(z)$ is given by \eqref{Eq5} through $h(z)=e^z$.
Then, $\rho(f)=\infty$, and for $c=\pi i$ and each $\alpha$ with $e^{\alpha c}=1$, $f^3(z)+f^3(z+c)=e^{\alpha z+\beta}$ for all $\beta\in\C$.
\end{example}

\begin{example}\label{Ex5}
Let $f_1(z)=e^{\frac{\alpha z+\beta}{2}}\sin(z)$ and $f_2(z)=e^{\frac{\alpha z+\beta}{2}}\sin(e^{4iz}+z)$.
Then, $\rho(f_1)\leq1$ and $\rho(f_2)=\infty$.
For $c=\frac{\pi}{2}$ and each $\alpha$ with $e^{\alpha c}=1$, $f^2_j(z)+f^2_j(z+c)=e^{\alpha z+\beta}$ for $j=1,2$ and all $\beta\in\C$.
\end{example}

\begin{example}\label{Ex6}
Define $f_1(z)=e^z+\frac{e^{\alpha z+\beta}}{2}$ and $f_2(z)=e^{e^{2z}+z}+\frac{e^{\alpha z+\beta}}{2}$.
Then, $\rho(f_1)\leq1$ and $\rho(f_2)=\infty$.
For $c=i\pi$ and each $\alpha$ with $e^{\alpha c}=1$, $f_j(z)+f_j(z+c)=e^{\alpha z+\beta}$ for $j=1,2$ and all $\beta\in\C$.
\end{example}

Opposite to theorem \ref{Th2}, even the existence of finite order or infinite order solutions to $f^2(z)+f^2(z+c)=e^{\alpha z+\beta}$ may be described for special $\alpha,c$, I was not able to characterize systematically all possible solutions to this difference equation.
The same concern occurs for the existence of infinite order solutions to $f^3(z)+f^3(z+c)=e^{\alpha z+\beta}$ in a systematic manner.

Finally, we briefly consider $f(z)+f(z+c)=e^{\alpha z+\beta}$.
Recall \eqref{Eq4} to choose $f(z)=de^{\alpha z+\beta}$ and $f(z+c)=e^{\alpha c}f(z)$.
When $d(1+e^{\alpha c})=1$, we are done.
I imagine the general solutions may be of the form $f(z)=\delta(z)+de^{\alpha z+\beta}$ for a meromorphic function $\delta(z)$ with $\delta(z+c)=-\delta(z)$ and $d(1+e^{\alpha c})=1$; in addition, the general solutions may be of the form $f(z)=\delta(z)-\frac{z}{c}e^{\alpha z+\beta}$ for $e^{\alpha c}=-1$.
When $n\geq2$, then equation \eqref{Eq4} may have no solution if $e^{\alpha c}=-1$.





\end{document}